\documentclass[letterpaper, 12pt]{amsart}

\newif\ifprinterfriendly\printerfriendlytrue
\printerfriendlyfalse
\usepackage[colorlinks = true, linkcolor = blue, citecolor = blue, pagebackref = true]{hyperref}

\usepackage[margin = 1in]{geometry}
\usepackage[backrefs, alphabetic]{amsrefs}
\usepackage{color}

\usepackage{mathtools} \usepackage{etoolbox}
\patchcmd{\section}{\scshape}{\bfseries}{}{} \makeatletter
\renewcommand{\@secnumfont}{\bfseries} \makeatother

\theoremstyle{plain} \newtheorem{theorem}{Theorem}
\theoremstyle{plain} \newtheorem{lemma}[theorem]{Lemma}
\theoremstyle{plain} 
\theoremstyle{plain} \newtheorem{corollary}[theorem]{Corollary}
\theoremstyle{definition} 
\theoremstyle{definition} \newtheorem*{remark*}{Remark}

\usepackage{color}
\usepackage[usenames,dvipsnames]{xcolor}

\newcommand{\allowcomments}[3]{
\newcommand{#1}[1]{\ifprinterfriendly{\textsl{ ##1 \--#2}}\else{\textcolor{#3}{##1}}\fi}
}

\allowcomments{\comfelipe}{FR}{magenta}
\allowcomments{\comdavid}{DS}{Green}
\allowcomments{\comfabian}{FS}{blue}

\newcommand{\NN}{\mathbb{N}}
\newcommand{\RR}{\mathbb{R}}
\newcommand{\QQ}{\mathbb{Q}}
\newcommand{\ZZ}{\mathbb{Z}}
\newcommand{\TT}{\mathbb{T}}
\newcommand{\cD}{\mathcal{D}}
\newcommand{\cQ}{\mathcal{Q}}
\newcommand{\cR}{\mathcal{R}}

\newcommand{\bq}{{\boldsymbol{q}}}
\newcommand{\br}{{\boldsymbol{r}}}
\newcommand{\bs}{{\boldsymbol{s}}}
\newcommand{\bx}{{\boldsymbol{x}}}
\newcommand{\bn}{{\boldsymbol{n}}}

\newcommand{\by}{{\boldsymbol{y}}}
\newcommand{\0}{{\boldsymbol{0}}}
\newcommand{\bgamma}{\boldsymbol{\gamma}} 

\def\eps{\varepsilon}

\providecommand{\ceil}[1]{\lceil#1\rceil}
\providecommand{\abs}[1]{\lvert#1\rvert}
\providecommand{\Abs}[1]{\left|#1\right|}
\providecommand{\norm}[1]{\lVert#1\rVert}

\providecommand{\inner}[2]{\langle#1,#2\rangle}

\title{Rational approximation of affine coordinate subspaces of Euclidean space} \author{Felipe~A.~Ram{\'i}rez \and David~S.~Simmons \and Fabian~S{\"u}ess}

\address{Wesleyan University, Department of Mathematics and Computer Science, 265 Church Street, Middletown, CT 06459}
\email{framirez@wesleyan.edu}

\address{University of York, Department of Mathematics, Heslington, York YO10 5DD, UK}
\email{David.Simmons@york.ac.uk}
\urladdr{\url{https://sites.google.com/site/davidsimmonsmath/}}

\address{University of York, Department of Mathematics, Heslington, York YO10 5DD, UK}
\email{fs752@york.ac.uk}

\date{}

\begin{document}

\begin{abstract}
  We show that affine coordinate subspaces of dimension at least two
  in Euclidean space are of Khintchine type for divergence. For affine
  coordinate subspaces of dimension one, we prove a result which
  depends on the dual Diophantine type of the basepoint of the
  subspace. These results provide evidence for the conjecture that all
  affine subspaces of Euclidean space are of Khintchine type for
  divergence.
\end{abstract}
\maketitle

\tableofcontents

\section{Introduction and results}
\label{sec:results}

The field of \emph{simultaneous Diophantine approximation} is
concerned with how well points in $\RR^d$ can be approximated by
points in $\QQ^d$. One way in which the quality of approximation is
measured is via some function $\psi:\NN\to\RR^+\cup\{0\}$ which we
often think of as a rate of approximation. Specifically, a vector
$\bx\in\RR^d$ is said to be \emph{$\psi$-approximable} if there are
infinitely many $q\in\NN$ such that $\norm{q\bx} < \psi(q)$, where
$\norm{\cdot}$ denotes sup-norm distance to the nearest integer
point. Given $\psi$, it is natural to wonder about the size of the set
of all $\psi$-approximable vectors. \emph{Khintchine's
  theorem}~\cite{Khintchine} answers this question when `size' means
Lebesgue measure. It is the foundational result of \emph{metric
  Diophantine approximation}, and it states that if $\psi$ is
nonincreasing, then the set of $\psi$-approximable vectors in $\RR^d$
has either zero measure or full measure, depending on whether the
series $\sum_{q\in\NN}\psi(q)^d$ converges or diverges,
respectively. Gallagher~\cite{Gallagherkt} proved that if $d\geq2$
then one can remove the monotonicity assumption on $\psi$.

While Khintchine's theorem gives us a means to determine the measure
of $\psi$-approximable vectors, it leaves us completely in the dark
regarding approximation within sets of zero measure. Such problems
arise very naturally. For instance, in two dimensions, Khintchine's
theorem tells us that if $\sum_{q\in\NN}\psi(q)^2$ diverges, then for
almost every $x\in\RR$ the set
$\{y\in\RR : (x,y) \text{ is $\psi$-approximable}\}\subseteq\RR$ has
full Lebesgue measure. However, \emph{a priori}, any given value of
$x$, for example $x = \sqrt 2$, may be an exception to this almost
everywhere statement. Ideally, we would like to obtain a statement of
the following form: Let $\ell$ and $k$ be positive integers with
$\ell+k=d$, and let $\psi$ be a nonincreasing approximating
function. Then, for $\bx\in\RR^\ell$, the set of $\psi$-approximable
vectors in $\{\bx\}\times\RR^k\subseteq\RR^d$ has either zero or full
$k$-dimensional Lebesgue measure, depending on whether the sum
$\sum_{q\in\NN}\psi(q)^d$ converges or diverges, respectively.

However, upon choosing a rational vector $\bx$, it is easily
established that the convergence part of the above statement cannot be
true. Hence, it is worth treating the different sides of the problem
separately. The convergence side has previously been addressed in
greater generality by Ghosh~\cite{Ghosh} (and will not be visited in
this paper). There, he shows that whether an affine subspace enjoys
the convergence property depends on the Diophantine type of the matrix
used to define the subspace.

An affine coordinate subspace $\{\bx\}\times\RR^k \subseteq\RR^d$ is
said to be of \emph{Khintchine type for divergence} if for any
nonincreasing function $\psi:\NN\to\RR^+$ such that
$\sum_{q\in\NN}\psi(q)^d$ diverges, almost every point on
$\{\bx\}\times\RR^k$ is $\psi$-approximable. Intuitively,
$\{\bx\}\times\RR^k$ is of Khintchine type for divergence if its typical
points behave like the typical points of Lebesgue measure with respect
to the divergence case of Khintchine's theorem. The recent
article~\cite{hyperplanes} addresses the issue for certain affine
coordinate hyperplanes in $\RR^d$, where $d\geq 3$. There, sufficient
conditions are given for a hyperplane to be of Khintchine type for
divergence. In this note we settle the case of affine coordinate
subspaces of dimension at least two and make vast progress on the
one-dimensional case.

\begin{remark*}
  It is worth noting that the notions of Khintchine types for
  convergence and divergence can be analogously defined for general
  manifolds. It has been proved by
  Beresnevich--Dickinson--Velani~\cite{BDVplanarcurves} and
  Beresnevich~\cite{Bermanifolds} that analytic nondegenerate
  submanifolds of $\RR^d$ are of Khintchine type for divergence. Here,
  an analytic submanifold of $\RR^d$ is said to be
  \emph{nondegenerate} if it is not contained in any affine
  hyperplane.

  In fact, it is conjectured that nondegenerate submanifolds of
  $\RR^d$ are also of Khintchine type for convergence, and this is
  known in several cases \cites{BVVZ, Simmons-convergence-case}. This
  contrasts with the situation for degenerate manifolds (e.g. affine
  subspaces), as the aforementioned results of Ghosh indicate. So it
  is interesting that in the divergence case, it does not seem to
  matter whether a manifold is degenerate or not.
\end{remark*}

Coming back to the present problem, we prove the following:

\begin{theorem}\label{thm:subspaces}
  Every affine coordinate subspace of Euclidean space of dimension at
  least two is of Khintchine type for divergence.
\end{theorem}

\begin{remark*}
  Combining Theorem \ref{thm:subspaces} with Fubini's theorem shows
  that every submanifold of Euclidean space which is foliated by
  affine coordinate subspaces of dimension at least two is of
  Khintchine type for divergence. For example, given $a,b,c\in\RR$
  with $(a,b)\neq (0,0)$, the three-dimensional affine subspace
  \[
    \{(x,y,z,w) : a x + b y = c\} \subseteq \RR^4
  \]
  is of Khintchine type for divergence, being foliated by the
  two-dimensional affine coordinate subspaces $(x,y)\times\RR^2$
  ($x,y\in\RR$, $a x + b y = c$).
\end{remark*}

The reason for the restriction to subspaces of dimension at least two
is that Gallagher's theorem is used in the proof, and it is only true
in dimensions at least two. Regarding one-dimensional affine
coordinate subspaces, we have the following weaker theorem:

\begin{theorem}\label{thm:lines}
  Consider a one-dimensional affine coordinate subspace
  $\{\bx\}\times\RR \subseteq \RR^d$, where $\bx\in\RR^{d - 1}$.
  \begin{itemize}
  \item[(i)] If the dual Diophantine type of $\bx$ is strictly greater
    than $d$, then $\{\bx\}\times \RR$ is contained in the set of very
    well approximable vectors
    \[
      \mathrm{VWA}_d = \{\by : \exists \eps > 0 \; \exists^\infty
      q\in\NN \;\; \norm{q\by} < q^{-1/d - \eps}\}.
    \]
  \item[(ii)] If the dual Diophantine type of $\bx$ is strictly less
    than $d$, then $\{\bx\}\times\RR$ is of Khintchine type for
    divergence.
  \end{itemize}
\end{theorem}

\noindent Here the \emph{dual Diophantine type} of a point
$\bx\in\RR^\ell$ is the number
\begin{equation}
  \label{taudef}
  \sup\left\{\tau\in\mathbb{R}^+: \norm{\inner{\bn}{\bx}} < |\bn|_\infty^{-\tau} \textrm{ for i.m. } \bn\in\ZZ^\ell\backslash\{\0\} \right\}.
\end{equation}
\begin{remark*}
  The inclusion $\{\bx\}\times\RR \subseteq \mathrm{VWA}_d$ in part
  (i) is philosophically ``almost as good'' as being of Khintchine
  type for divergence, since it implies that for sufficiently ``nice''
  functions $\psi:\NN\to\RR^+$ such that $\sum_{q\in\NN} \psi(q)^d$
  diverges, almost every point on $\{\bx\}\times\RR$ is
  $\psi$-approximable. For example, call a function $\psi$ \emph{good}
  if for each $c > 0$, we have either $\psi(q)\geq q^{-c}$ for all $q$
  sufficiently large or $\psi(q) \leq q^{-c}$ for all $q$ sufficiently
  large. Then by the comparison test, if $\psi$ is a good function
  such that $\sum_{q\in\NN} \psi(q)^d$ diverges, then for all
  $\eps > 0$, we have $\psi(q) \geq q^{-1/d - \eps}$ for all $q$
  sufficiently large, and thus by Theorem \ref{thm:lines}(i), every
  point of $\{\bx\}\times\RR$ is $\psi$-approximable. The class of
  good functions includes the class of \emph{Hardy $L$-functions}
  (those that can be written using the symbols $+,-,\times,\div,\exp$,
  and $\log$ together with the constants and the identity function)
  \cite[Chapter III]{Hardy}; {\it cf.} \cite{AvdD} for further
  discussion and examples.
\end{remark*}

Taken together, parts (i) and (ii) of Theorem \ref{thm:lines} imply
that if $\psi$ is a Hardy $L$-function such that
$\sum_{q\in\NN} \psi(q)^d$ diverges, and if $\bx\in\RR$ is a vector
whose dual Diophantine type is not exactly equal to $d$, then almost
every point of $\{\bx\}\times\RR \subseteq \RR^d$ is
$\psi$-approximable. This situation is somewhat frustrating, since it
seems strange that points in $\RR^{d - 1}$ with dual Diophantine type
exactly equal to $d$ should have any special properties (as opposed to
those with dual Diophantine type $(d - 1)$, which are the ``not very
well approximable'' points). However, it seems to be impossible to
handle these points using our techniques.

{\bf Acknowledgements.} The authors were supported in part by the
EPSRC Programme Grant EP/J018260/1. We thank Victor Beresnevich, Jon Chaika, and Sanju Velani for helpful
discussions and advice.

\section{Proof of Theorem~\ref{thm:subspaces}: Subspaces of dimension
  at least two}\label{sec:gallagher}

Consider an affine coordinate subspace $\{\bx\}\times\RR^k$, where
$\bx\in\RR^\ell$ and $\ell + k = d$. Given a nonincreasing function
$\psi:\NN\to\RR^+$, for each $M,N$ with $M < N$ let
\begin{equation*}
  Q(M,N): = Q_{\bx,\psi}(M,N) = \Abs{\left\{M < q\leq N : \norm{q\bx} < \psi(N) \right\}},
\end{equation*}
and write $Q(N) : = Q(0, N)$. Since any real number $\delta > 0$ may
be thought of as a constant function, the expression $Q_\delta(M,N)$
makes sense. Note that $Q_\psi(M,N) = Q_{\psi(N)}(M,N)$.

\begin{lemma}\label{lem:count}
  For all $N\in\NN$,
  \begin{equation*}
    Q_\delta(N) = \Abs{\left\{q\in\NN : \norm{q\bx} < \delta, q\leq N\right\}}\geq N\delta^\ell - 1.
  \end{equation*}
\end{lemma}

\begin{proof}
  Let
  $\cQ_\delta(N) = \left\{q\in\NN : \norm{q\bx} < \delta, q\leq
    N\right\}$, so that $Q_\delta = \Abs{\cQ_\delta}$.

  We first claim that
  $Q_\delta(N)\geq Q_{\frac{\delta}{2},\bgamma}(N) - 1$ for any
  $\bgamma\in\RR^\ell$ and $N\in\NN$, where
  \begin{equation*}
    \cQ_{\delta,\bgamma} : = \left\{q\in\NN : \norm{q\bx + \bgamma} < \delta, q\leq N\right\}
  \end{equation*}
  and $Q_{\delta,\bgamma}(N) = \Abs{\cQ_{\delta,\bgamma}(N)}$. Simply
  notice that if $q_1 < q_2\in\cQ_{\frac{\delta}{2},\bgamma}(N)$, then
  by the triangle inequality, $q_2 -
  q_1\in\cQ_{\delta}(N)$.
  Therefore, letting $q_0 = \min\cQ_{\frac{\delta}{2},\bgamma}(N)$, we
  have
  $\cQ_{\frac{\delta}{2},\bgamma}(N) -
  q_0\subseteq\cQ_\delta(N)\cup\{0\}$,
  which implies that $Q_\delta(N)\geq Q_{\delta/2,\bgamma}(N) - 1$.

  Now we show that for any $N\in\NN$ there is some $\bgamma$ such that
  $Q_{\frac{\delta}{2},\bgamma}(N)\geq N\delta^\ell$. Notice that
  \begin{align*}
    \int_{\TT^\ell} Q_{\frac{\delta}{2},\bgamma}(N)\,d\bgamma & = \int_{\TT^\ell} \sum_{q = 1}^N \mathbf{1}_{\left( - \frac{\delta}{2}, \frac{\delta}{2}\right)^\ell}(q\bx + \bgamma)\,d\bgamma = N\delta^\ell.
  \end{align*}
  Therefore, $Q_{\frac{\delta}{2},\bgamma}(N)$ must take some value
  $\geq N \delta^\ell$ at some $\bgamma$. Combining this with the
  previous paragraph proves the lemma.
\end{proof}

\begin{lemma}
  \label{lem:series}
  Suppose that $\sum_{q\in\NN}\psi(q)^d$ diverges. Then
  \begin{equation}
    \label{psikseries}
    \sum_{\norm{q\bx} < \psi(q)}\psi(q)^{k} = \infty.
  \end{equation}
\end{lemma}

\begin{proof}
  We may assume without loss of generality that $\psi(q) = 2^{-m_q}$
  where $m_q\in\NN$. (Indeed, given any $\psi$ as in the theorem
  statement, we can let $m_q = \ceil{-\log_2\psi(q)}$ and replace
  $\psi(q)$ with $2^{-m_q}$. We will have changed $\psi$ by no more
  than a factor of $\frac{1}{2}$, preserving the divergence of the
  series $\sum_{q\in\NN}\psi(q)^d$, but since the new $\psi$ is less
  than the old $\psi$, divergence of \eqref{psikseries} for the new
  $\psi$ implies divergence of \eqref{psikseries} for the old $\psi$.)
  Now,
  \begin{align*}
    \sum_{\norm{q\bx} < \psi(q)}\psi(q)^{k} &\geq \sum_{m\in\NN}\psi(2^m)^{k}\Abs{\left\{2^{m - 1} < q\leq 2^m : \norm{q\bx} < \psi(2^m)\right\}}\\
                                            & = \sum_{m\in\NN}\psi(2^m)^{k}Q(2^{m - 1}, 2^m) \\
                                            & = \sum_{m\in\NN}\sum_{n\geq m}\left(\psi(2^n)^{k} - \psi(2^{n + 1})^{k}\right)Q(2^{m - 1}, 2^m)\\
                                            & = \sum_{n\in\NN}\left(\psi(2^n)^{k} - \psi(2^{n + 1})^{k}\right)\sum_{m = 1}^n Q(2^{m - 1}, 2^m) \\
                                            &\geq \sum_{n\in\NN}\left(\psi(2^n)^{k} - \psi(2^{n + 1})^{k}\right) Q(2^n) \\
                                            &\geq \sum_{n\in\NN}\left(\psi(2^n)^{k} - \psi(2^{n + 1})^{k}\right)
                                              [2^n\psi(2^n)^\ell - 1] \;\;\;\; (\text{Lemma \ref{lem:count}})\\
                                            &=  -\psi(1)^k + \sum_{n\in\NN}\left(\psi(2^n)^{k} - \psi(2^{n + 1})^{k}\right)
                                              2^n\psi(2^n)^\ell.
  \end{align*}
  Now, let $(n_j)_{j = 1}^\infty$ be the sequence indexing the set
  $\{n\in\NN : m_{2^{n}}\neq m_{2^{n + 1}}\}$ in increasing
  order. Then we have
  $\psi(2^{n_j})^k - \psi(2^{n_j + 1})^k \gg \psi(2^{n_j})^k$, where
  $\gg$ is the Vinogradov symbol. So
  \begin{align*}
    \sum_{n\in\NN}\left(\psi(2^n)^{k} - \psi(2^{n + 1})^{k}\right)
    2^n\psi(2^n)^\ell
    &\gg \sum_{j\in\NN}2^{n_j} \psi(2^{n_j})^{k + \ell} \\
    &\gg \sum_{j\in\NN}\left(\sum_{m = n_{j - 1} + 1}^{n_j}2^m\right)\psi(2^{n_j})^d \\
    & = \sum_{m\in\NN}2^m\psi(2^m)^{d},
  \end{align*}
  which diverges by Cauchy's condensation test.
\end{proof}

\begin{proof}[Proof of Theorem~\ref{thm:subspaces}]
  Suppose that $k\geq 2$. Then by Lemma \ref{lem:series}, we can apply
  Gallagher's extension of Khintchine's theorem \cite{Gallagherkt} to
  the function
  \begin{equation}
    \label{psix}
    \psi_{\bx}(q) = \begin{cases}
      \psi(q) & \norm{q\bx} < \psi(q) \\
      0 & \text{otherwise}
    \end{cases}
  \end{equation}
  and get that $\{\bx\}\times\RR^k$ is of Khintchine type for
  divergence. But $\bx\in\RR^\ell$ was arbitrary, and applying
  permutation matrices does not affect whether a manifold is of
  Khintchine type for divergence. This completes the proof.
\end{proof}

\section{Proof of Theorem~\ref{thm:lines}(i): Base points of high
  Diophantine type}\label{sec:ktp}

The proof of Theorem \ref{thm:lines}(i) is based on the following
standard fact, which can be found for example in~\cite[Theorem
V.IV]{Cassels}:

\theoremstyle{plain} \newtheorem*{ktp}{Khintchine's transference
  principle}
\begin{ktp}
  Let $\bx\in\RR^d$ and define the numbers
  \begin{equation*}
    \omega_D: = \omega_D(\bx) = \sup\left\{\omega\in\mathbb{R}^+: \norm{\inner{\bn}{\bx}} \leq \abs{\boldsymbol{n}}_\infty^{-(d + \omega)} \textrm{ for i.m. } \bn\in\ZZ^d\backslash\{0\} \right\}.
  \end{equation*}
  and
  \begin{equation*}
    \omega_S : = \omega_S(\bx) = \sup\left\{\omega\in\mathbb{R}^+: \norm{q\boldsymbol{x}} \leq q^{-(1 + \omega)/d} \textrm{ for i.m. } q\in\NN \right\}.
  \end{equation*}
  Then
  \[
    \frac{\omega_D}{d^2 + (d - 1)\omega_D} \leq \omega_S \leq \omega_D
  \]
  where the cases $\omega_D = \infty$ and $\omega_S = \infty$ should
  be interpreted in the obvious way.
\end{ktp}

Note that $\omega_D$ is related to the dual Diophantine type $\tau_D$
defined in \eqref{taudef} via the formula
$\tau_D(\bx) = \omega_D(\bx) + d$.

\begin{proof}
[Proof of Theorem~\ref{thm:lines}(i)]
We fix $\bx = (x_1,\dots, x_{d - 1})\in\RR^{d - 1}$ such that
$\tau_D(\bx) > d$, and we consider a point
$(\bx,y) \in \{\bx\}\times\RR$. It is clear from \eqref{taudef} that
$\tau_D(\bx,y) \geq \tau_D(\bx)$, so $\tau_D(\bx,y) > d$ and thus
$\omega_D(\bx,y) > 0$. Thus by Khintchine's transference principle,
$\omega_S(\bx,y) > 0$, {\it i.e.} $(\bx,y) \in \mathrm{VWA}_d$.
\end{proof}

\section{Proof of Theorem~\ref{thm:lines}(ii): Base points of low
  Diophantine type}\label{sec:ubiquity}

Let $\psi:\NN\to\RR^+$ be nonincreasing and such that
$\sum_{q\in\NN}\psi(q)^d$ diverges. Our goal here is to use the ideas
of ubiquity theory to show that almost every point on
$\{\bx\}\times\RR\subseteq\RR^d$ is $\psi$-approximable, where
$\bx\in\RR^{d - 1}$ has been fixed with dual Diophantine type strictly
less than $d$. The ubiquity approach begins with the fact that for any
$N\in\NN$ such that
\begin{equation}
  \label{Nreq}
  N^{-1/(d - 1)} < \psi(N) < 1,
\end{equation}
we have
\begin{equation}
  \label{eq:mink}
  [0,1] \subseteq \bigcup_{\substack{q\leq N \\ \norm{q \bx} < \psi(N)}}\bigcup_{p = 0}^q B\left(\frac{p}{q},\frac{2}{q N\psi(N)^{d - 1}}\right),
\end{equation}
which is a simple consequence of Minkowski's theorem. The basic aim is
to show that a significant proportion of the measure of the above
double-union set is represented by $q$s that are closer to $N$ than to
$0$. Specifically, we must show that for some $k\geq 2$, the following
three things hold:
\begin{enumerate}
\item[\textbf{(U)}] The pair $(\cR,\beta)$ forms a (global)
  \textbf{ubiquitous system} with respect to the triple $(\rho,l,u)$,
  where
  \begin{align*}
    J &= \{p/q \in \QQ : \norm{q\bx} < \psi(q)\},&
                                                   R_{p/q} &= \{p/q\} \;\;\; (p/q \in J),\\
    \cR &= \{R_{p/q} : p/q\in J\},&
                                    \beta_{p/q} &= q \;\;\; (p/q\in J),\\
    l_j &= k^{j - 1} \;\;\; (j\in\NN),&
                                        u_j &= k^j \;\;\; (j\in\NN),\\
    \rho(q) &= \frac{c}{q^2 \psi(q)^{d - 1}} \;\;\; (q\in\NN),
  \end{align*}
  and $c > 0$ will be chosen later. This means that there is some
  $\kappa > 0$ such that
  \begin{equation*}
    \lambda\left([0,1]\cap\bigcup_{\substack{k^{j - 1} < q \leq k^j \\ \norm{q\bx} < \psi(k^j)}}\bigcup_{p = 0}^q
      B\left(\frac{p}{q},\frac{c}{k^{2j}\psi(k^j)^{d - 1}}\right)\right) \geq \kappa
  \end{equation*}
  for all $j$ sufficiently large. \item[\textbf{(R)}] The function
  $\Psi(q): = \psi(q)/q$ is \textbf{$u$-regular}, meaning that there
  is some constant $\kappa < 1$ such that
  $\Psi(k^{j + 1})\leq \kappa \Psi(k^j)$ for all $j$ sufficiently
  large.
\item[\textbf{(D)}] The sum
  $\sum_{j\in\NN}\frac{\Psi(k^j)}{\rho(k^j)}$ \textbf{diverges}.
\end{enumerate}
Then~\cite[Corollary~2]{limsup} will imply that the set of
$\psi_\bx$-approximable numbers (see \eqref{psix}) in $\RR$ has
positive measure, and Cassels' ``0-1 law'' \cite{Cassels-01law} will
imply that it has \emph{full} measure. Since the set of
$\psi_\bx$-approximable numbers is just the set of $y\in\RR$ for which
$(\bx,y)$ is $\psi$-approximable, this will show that the set of
$\psi$-approximable points on the line
$\{\bx\}\times\RR\subseteq\RR^d$ has full (one-dimensional Lebesgue)
measure.

The following lemma shows that \textbf{(R)} and~\textbf{(D)} are easy.

\begin{lemma}\label{lem:RD}
  If $\psi:\NN\to\RR^+$ is nonincreasing, then~\textup{\textbf{(R)}}
  holds. Furthermore, if $\sum_{q\in\NN}\psi(q)^d$ diverges,
  then~\textup{\textbf{(D)}} holds.
\end{lemma}

\begin{proof}
  In the first place, we have
  \begin{equation*}
    \frac{\Psi(k^{j + 1})}{\Psi(k^j)} = \frac{\psi(k^{j + 1})}{k\psi(k^j)}
    \leq \frac{1}{k},
  \end{equation*}
  which proves~\textbf{(R)}. For~\textbf{(D)},
  \begin{equation*}
    \sum_{j\in\NN}\frac{\Psi(k^j)}{\rho(k^j)} = \sum_{j\in\NN}k^j\psi(k^j)^d,
  \end{equation*}
  which diverges by Cauchy's condensation test.
\end{proof}

The challenge then is to establish~\textbf{(U)}.

\begin{lemma}\label{ubiqlemma}
  Let $\psi:\NN\to\RR^+$ be nonincreasing such that \eqref{Nreq} holds
  for all sufficiently large $N$, and assume that for all $k\geq 2$,
  there exists $j_k\geq 1$ such that for all $j\geq j_k$,
  \begin{equation*} \Abs{\left\{0 < q\leq k^{j - 1} : \norm{q\bx} <
        \psi(k^j)\right\}} \ll k^{j - 1}\psi(k^j)^{d - 1},
  \end{equation*}
  where $\ll$ is the Vinogradov symbol, whose implied constant is
  assumed to be independent of $k$. Then~\textup{\textbf{(U)}} holds
  for some $k\geq 2$.
\end{lemma}

\begin{proof}
  For all $k\geq 2$ and $j\geq j_k$, we have
  \begin{equation*}
    \lambda\left([0,1]\cap\bigcup_{\substack{q\leq k^{j - 1} \\ \norm{q\bx} < \psi(k^j)}}\bigcup_{p = 0}^q B\left(\frac{p}{q},\frac{2}{q k^j \psi(k^j)^{d - 1}}\right)\right)
    \leq\sum_{\substack{q\leq k^{j - 1} \\ \norm{q\bx} < \psi(k^j)}}\frac{4}{k^j\psi(k^j)^{d - 1}}
    \ll\frac{1}{k}\cdot
  \end{equation*}
  After choosing $k$ to be larger than the implied constant in the
  ``$\ll$'' comparison, we see that the left hand side is
  $\leq 1 - \kappa < 1$ for some $\kappa > 0$.

  Combining with \eqref{eq:mink}, we see that for all $j\geq j_k$
  large enough so that \eqref{Nreq} holds for $N = k^j$, we have
  \begin{equation*}
    \lambda\left([0,1]\cap\bigcup_{\substack{k^{j - 1} < q\leq k^j \\ \norm{q\bx} < \psi(k^j)}}\bigcup_{p = 0}^q
      B\left(\frac{p}{q},\frac{2}{q k^j \psi(k^j)^{d - 1}}\right)\right)
    \geq \kappa > 0,
  \end{equation*}
  and this implies \textbf{(U)} with $c = 2k$.
\end{proof}

The one-dimensional case of the following lemma was originally proven
by Beresnevich, Haynes, and Velani using a continued fraction
argument. This argument will appear in a forthcoming paper of theirs,
currently in preparation \cite{nalpha}.

\begin{lemma}
  \label{nalphalemma}
  Fix $\bx\in\RR^\ell$ and $\tau > \tau_D(\bx)$. Then for all $N$
  sufficiently large and for all $\delta \geq N^{-1/\tau}$, we have
  \begin{equation}
    \label{BHVformula}
    \Abs{\{q\in\mathbb{N}:\norm{q\bx} < \delta,q\leq N\}}\leq 4^{\ell + 1} N \delta^\ell.
  \end{equation}
\end{lemma}
\begin{proof}
  Consider the lattice $\Lambda = g_t u_\bx \ZZ^{\ell + 1}$, where
  \begin{align*}
    g_t & = \left[\begin{array}{ll}
                    e^{t/\ell}I_\ell &\\
                                     & e^{-t}
                  \end{array}\right]\\
    u_\bx & = \left[\begin{array}{ll}
                      I_\ell & - \bx\\
                             & 1
                    \end{array}\right],
  \end{align*}
  where $t$ is chosen so that $R : = e^{t/\ell} \delta = e^{-t} N$,
  {\it i.e.} $t = \log(N/\delta)/(1 + 1/\ell)$. Then \eqref{BHVformula} can
  be rewritten as
  \[
    \Abs{\{\br\in \Lambda : |\br|_\infty < R\}} \leq (4R)^{\ell + 1}.
  \]
  Now let $\cD$ be the Dirichlet fundamental domain for $\Lambda$
  centered at $\0$, {\it i.e.}
  \[
    \cD = \{\br\in\RR^{\ell + 1} : \mathrm{dist}(\br,\Lambda) =
    \mathrm{dist}(\br,\0) = |\br|_\infty\}.
  \]
  Since $\Lambda$ is unimodular, $\cD$ is of volume 1, so
  \[
    \Abs{\{\br\in \Lambda : |\br|_\infty < R\}} =
    \lambda\left(\bigcup_{\substack{\br\in\Lambda \\ |\br|_\infty <
          R}} (\br + \cD)\right) \leq^* \lambda(B_{\ell + 1}(\0,2R)) =
    (4R)^{\ell + 1},
  \]
  where the starred inequality is true as long as
  $\cD \subseteq B_{\ell + 1}(\0,R)$. So we need to show that
  $\cD\subseteq B_{\ell + 1}(\0,R)$ assuming that $N$ is large enough.

  Suppose that $\cD \not\subseteq B_{\ell + 1}(\0,R)$. Then the last
  successive minimum of $\Lambda$ is $\gg R$, so by \cite[Theorem
  VIII.5.VI]{Cassels-geometry}, some point $\bs$ in the dual lattice
  $\Lambda^* = \{\bs\in \RR^{\ell + 1} : \inner{\br}{\bs}\in\ZZ\}$
  satisfies $0 < |\bs|_\infty \ll R^{-1}$. Write
  $\bs = g_t' u_\bx' (\bq,p)$ for some $p\in\ZZ$, $\bq\in\ZZ^\ell$,
  where $g_t'$ and $u_\bx'$ denote the inverse transposes of $g_t$ and
  $u_\bx$, respectively. Then the inequality $|\bs|_\infty \ll R^{-1}$
  becomes
  \begin{align*}
    \begin{split}
      e^{-t/\ell} |\bq|_\infty &\ll R^{-1}\\
      e^t |\inner\bq\bx + p| &\ll R^{-1}
    \end{split}
                               \hspace{.6in}\text{ i.e. }
                               \begin{split}
                                 |\bq|_\infty &\ll \delta^{-1}\\
                                 |\inner\bq\bx + p| &\ll N^{-1}.
                               \end{split}
  \end{align*}
  Since $\delta \geq N^{-1/\tau}$ we get
  \begin{equation}
    \label{final}
    |\inner\bq\bx + p| \ll \delta^\tau \ll |\bq|_\infty^{-\tau}.
  \end{equation}
  Since $\tau > \tau_D(\bx)$, there are only finitely many pairs
  $(p,\bq)$ satisfying \eqref{final}. Thus for all sufficiently large
  $N$, we have $\cD \subseteq B_{\ell + 1}(\0,R)$ and thus
  \eqref{BHVformula} holds.
\end{proof}

From this we can deduce the following consequence.

\begin{corollary}\label{cor:nalpha}
  Let $\bx\in\RR^{d - 1}$ be of dual Diophantine type
  $\tau_D(\bx) < d$ and suppose that for all $\eps > 0$, we have
  $\psi(q) \geq q^{-1/d - \eps}$ for all $q$ sufficiently large. Then
  for any $k\geq 2$ and $\ell\in\ZZ$, we have
  \begin{equation*}
    \Abs{\left\{ 0 < q \leq k^{j + \ell} : \norm{q\bx} < \psi(k^j) \right\}}\ll k^{j + \ell}\psi(k^j)^{d - 1}
  \end{equation*}
  for $j$ large enough.
\end{corollary}

\begin{proof}
  We show that for large enough $j$ we are in a situation where we can
  apply Lemma~\ref{nalphalemma}, with $N = k^{j + \ell}$ and
  $\delta = \psi(k^j)$. Since $\tau_D < d$ we can choose
  $\tau\in (\tau_D,d)$ and then for all large enough $j$,
  \begin{equation*}
    N^{-1/\tau} = k^{-(j + \ell)/\tau} < \psi(k^j);
  \end{equation*}
  hence Lemma~\ref{nalphalemma} applies.
\end{proof}

We are now ready for the last proof of this paper.

\begin{proof}
[Proof of Theorem~\ref{thm:lines}(ii)]
Let $\bx\in\RR^{d - 1}$ be a point whose dual Diophantine type is
strictly less than $d$, and let $\psi:\NN\to\RR^+$ be a nonincreasing
function such that $\sum_{q\in\NN}\psi(q)^d$ diverges. Furthermore,
assume that for every $\eps > 0$, the inequality
$1 > \psi(q) \geq q^{-1/d - \eps}$ holds for all sufficiently large
$q$. Then by Corollary~\ref{cor:nalpha}, we satisfy all the parts of
Lemma~\ref{ubiqlemma}, so there exists $k\geq 2$ such that
\textbf{(U)} holds. Thus by the argument given earlier, we can use
\cite[Corollary~2]{limsup} to conclude that almost every point on the
line $\{\bx\}\times\RR\subseteq\RR^d$ is $\psi$-approximable.

We now show that that the assumption that
$1 > \psi(q) \geq q^{-1/d - \eps}$ (for large enough $q$) can be made
without losing generality. If $\psi(q) \geq 1$ for all $q$, then all
points are $\psi$-approximable and the theorem is trivial, and if
$\psi(q) < 1$ for some $q$, then by monotonicity $\psi(q) < 1$ for all
$q$ sufficiently large. So we just need to show that the assumption
$\psi(q) \geq q^{-1/d - \eps}$ can be made without loss of
generality. Let $\phi(q) = (q(\log(q))^2)^{-1/d}$ and define the
function $\overline\psi(q) = \max\{\psi(q),\phi(q)\}$. Then
$\overline\psi$ satisfies our assumptions, and therefore almost every
point on $\{\bx\}\times\RR$ is $\overline\psi$-approximable. But
\begin{equation*}
  \sum_{\norm{q\bx} < \phi(q)}\phi(q) \leq \sum_{j\in\NN}\phi(2^j)\Abs{\left\{0 < q\leq 2^{j + 1} : \norm{q\bx} < \phi(2^j)\right\}}
  \overset{\textrm{Cor.}~\ref{cor:nalpha}}{\ll} \sum_{j\in\NN}2^{j + 1} \phi(2^j)^d,
\end{equation*}
which converges because $\sum_{q\in\NN}\phi(q)^d$ does, and therefore
almost every point on $\{\bx\}\times\RR$ is not
$\phi$-approximable. But every $\overline\psi$-approximable point
which is not $\phi$-approximable is $\psi$-approximable. Therefore,
the set of $\psi$-approximable points on the line
$\{\bx\}\times\RR\subseteq\RR^d$ is of full measure, and the theorem
is proved.
\end{proof}

\end{document}